\documentclass[12pt,reqno]{amsart}
\usepackage{amsthm,amsfonts,amsmath,amssymb,amsaddr,geometry}

\usepackage{url}
\usepackage{algorithm}
\usepackage{algorithmic}
\usepackage{subfigure}
\usepackage{tikz}
\usepackage{hyperref}
\hypersetup{
    colorlinks=true,
    linkcolor=green,
    }
\numberwithin{equation}{section}

\usepackage{amsmath}
\graphicspath{{Fig/}}
\usepackage{algorithm}
\usepackage{subfigure}
\usepackage{tikz}
\newcommand{\maK}{\mathcal K}
\newcommand{\maT}{\mathcal T}
\newcommand{\bm}[1]{\mbox{\boldmath$#1$}}
\def\be{\begin{equation}}
\def\ee{\end{equation}}

\def\bal{\begin{aligned}}
\def\eal{\end{aligned}}
\def\bes{\begin{equation*}}
\def\ees{\end{equation*}}
\newcommand{\pa}{\partial}

\newtheorem{theorem}{Theorem}[section]

\newtheorem{lemma}[theorem]{Lemma}
\newtheorem{lem}[theorem]{Lemma}

\newtheorem{definition}[theorem]{Definition}
\newtheorem{remark}[theorem]{Remark}
\newtheorem{example}[theorem]{Example}

\usepackage{blindtext}
\usepackage{lastpage}
\usepackage{fancyhdr}
\begin{document}
\graphicspath{{Fig/}}

\title{A GRADED MESH REFINEMENT FOR 2D POISSON'S EQUATION ON NON-CONVEX POLYGONAL DOMAINS}
\author{Charuka D. Wickramasinghe $^1$, Priyanka Ahire$^{2,*}$}
\maketitle

\begin{center}
{
\footnotesize $^1$Karmanos Cancer Institute, Wayne State University, School of Medicine, Detroit, MI, 48201, United States \\
$^2$Department of Mathematics, Wayne State University, Detroit, MI, 48202, United States\\
$^*$Corresponding author:gn6587@wayne.edu\\

%Received Jan. 1, 2019
}
\end{center}

\begin{abstract}
This work delves into solving the two-dimensional Poisson problem through the Finite Element Method which is relevant in various physical scenarios including heat conduction, electrostatics, gravity potential, and fluid dynamics. However, finding exact solutions to these problems can be complicated and challenging due to complexities in the domains such as re-entrant corners, cracks, and discontinuities of the solution along the boundaries, and due to the singular source function $f$. Our focus in this work is to solve the Poisson equation in the presence of re-entrant corners at the vertices of $\Omega$ where some of the interior angles are greater than $\pi$.  When the domain features a re-entrant corner, the numerical solution can display singular behavior near the corners. To address this, we propose a graded mesh algorithm that helps us to tackle the solution near singular points. We derive $H^1$ and $L^2$ error estimate results, and we use MATLAB to present numerical results that validate our theoretical findings. By exploring these concepts, we hope to provide new insights into the Poisson problem and inspire future research into the application of numerical methods to solve complex physical scenarios.


\end{abstract}

\vspace{10mm}

\section{Introduction}
In this paper, we consider the following stationary state Poisson equation with Dirichlet boundary condition

\begin{equation}\label{e1}
  -\Delta u  = f  \hspace{0.3cm} \mbox{in} \hspace{0.3cm} \Omega \hspace{1.3cm} u  = 0   \hspace{0.3cm} \mbox{on} \hspace{0.3cm} \partial \Omega  
\end{equation}

where, the Laplace operator $\Delta = \frac{\partial^2}{\partial x^2}+ \frac{\partial^2}{\partial y^2} $ and $\Omega$ is a bounded polygonal domain. In this work, first,  we validate existing theoretical results by solving 2D Poisson equation using linear finite elements for convex domains. The main work in this research is to present a graded mesh algorithm that enables us to capture the singular behavior of the numerical solution due to re-entrant corners on non-convex domains. Then we solve the 2D Poisson equation using the finite element method which is a widely used numerical technique for solving differential equations that arise in mathematical modeling and engineering.

The Poisson problem has applications in engineering and applied mathematics including heat conduction, electrostatics, gravity potential, fluid dynamics, and many other fields. However, solving these problems numerically presents major computational difficulties due to complexities in the domains such as re-entrant corners, cracks, and discontinuities of the solution along the boundaries, and due to the singular source function $f$.  Even when an exact solution can be obtained, a numerical solution may be preferable, especially if the exact solution is very complicated. Our focus in this work is to solve the Poisson equation in the presence of re-entrant corners at the vertices of $\Omega$ where some of the interior angles are greater than $\pi$. 

By the regularity theory, the solution $u$ is in $H^{1+\beta}(\Omega)$ with the regularity index $\beta = min (\frac{\pi}{\alpha_{i}}, 1)$ , where $\alpha_{i}$ are interior angles of the polygonal domain $\Omega$. It is easy to see that when the maximum interior angle
is larger than $\pi$, i.e., $\Omega$ is non-convex, $u \not\in H^{2}(\Omega) $ and thus the finite element approximation based on quasi-uniform grids will not produce the optimal convergence rate. Graded meshes near the singular vertices are employed to recover the optimal convergence rate. Such meshes can be
constructed based on a priori estimates \cite{Ivo}, \cite{Babuska}, \cite{Bacuta}, \cite{YUN}, \cite{Huang}, \cite{HLi}, \cite{Raugel} or on a posteriori analysis \cite{Binev}, \cite{Cascon}, \cite{Stevenson}.
In this paper, we shall consider the approach used in \cite{Bacuta},\cite{HLi},\cite{charu1a} and in particular, focus on the linear finite element approximation of (\ref{e1}).

Instead of standard Sobolev spaces, we here use weighted Sobolev spaces to prove the results
on graded meshes for corner singularities. 
In \cite{Zi-Cai}, \cite{LiZC}, knowledge of singular expansions of the solution near the vertices is used to prove super convergence on rectangular meshes. Also in \cite{Wu}, the knowledge of singular expansions of the solution near the vertices is
used to justify the super-convergence of recovered gradients on adaptive grids obtained from a
posteriori processes.
 We use weighted Sobolev spaces to prove the supe- convergence of the solution on a class of graded meshes for corner singularities, which can be
generated by a simple and explicit process. Since the singular expansion is not required in our
analysis, it is possible to extend our results to other singular problems (transmission problems,
Schrodinger type operators, and many other singular operators from physics) \cite{HLi}, \cite{Hengguang}, which can be treated in similar weighted Sobolev spaces.
Throughout this paper, by $x\lesssim y$, we mean $x\leq Cy$, for a generic constant $C > 0$, and by
$x \simeq y$, we mean $x\lesssim y$  and $y\lesssim x$. All constants hidden in this notation are independent of
the problem size N and the solution. However, they may depend on the shape of $\Omega$, and on
other parameters which will be specified in the context.

\begin{remark}
  For simplicity, the current paper focuses on analyzing a 2-dimensional Poisson problem with linear finite elements. However, the analysis could be extended to 3 dimensions and higher-order finite elements although this may present some challenges. Additionally, the problem could be expanded to include non-homogeneous boundary conditions through a simple linear transformation.  
\end{remark}

The rest of the article is organized as follows: In section 2, we present the standard finite element method and  $H^1$ and $L^2$ error estimate results for the 2D Poisson equation under a convex domain. In section 3, we introduced weighted Sobolev spaces and a graded mesh algorithm to solve the Poisson equation on non-convex domains using linear finite elements. We also present the $H^1$ and $L^2$ error estimate results for non-convex domains. In section 4, we present numerical results to validate our theoretical results and a conclusion is the section 5. Throughout the following text, the generic positive constants $C$ may take different values in different formulas but it is always independent of the mesh.
The rest of the article is organized as follows: In section 2, we present the standard finite element method and  $H^1$ and $L^2$ error estimate results for the 2D Poisson equation under a convex domain. In section 3, we introduced weighted Sobolev spaces and a graded mesh algorithm to solve the Poisson equation on non-convex domains using linear finite elements. We also present the $H^1$ and $L^2$ error estimate results for non-convex domains. In section 4, we present numerical results to validate our theoretical results and a conclusion is the section 5. Throughout the following text, the generic positive constants $C$ may take different values in different formulas but it is always independent of the mesh.

\section{Finite element method}
In this section, we will present a basic finite element algorithm, its well-posedness, and its regularity for Poisson's equation. We also present $H^1$ and $L^2$ error estimate results for the 2D Poisson equation (\ref{e1}) for convex polygonal domains for linear finite elements. 

\subsection{Finite Element Algorithm}
The Poisson equation under consideration is as follows: Let $\Omega\subset \mathbb R^2$ be a polygonal domain. Consider the Poisson problem
\begin{eqnarray}\label{ch2poi}
-\Delta u=f \quad {\rm{in}} \ \Omega,\qquad \quad u=0 \quad {\rm{on}}  \ \pa\Omega,
\end{eqnarray}

We denote by $H^m(\Omega)$ for an integer $m\geq 0$, the Sobolev space that consists of square-integrable functions whose $i$th weak derivatives are also square-integrable for $0\leq i\leq m$. For $s>0$ that is not an integer, we denote by $H^s(\Omega)$ the fractional order Sobolev space. For $\tau\geq 0$, $H_0^\tau(\Omega)$ represents the closure in $H^\tau(\Omega)$ of the space of $C^\infty$ functions with compact supports in $\Omega$, and $H^{-\tau}(\Omega)$ represents the dual space of $H_0^\tau(\Omega)$.
Let $L^2(\Omega):=H^0(\Omega)$. We shall denote the norm $\|\cdot\|_{L^2(\Omega)}$ by $\|\cdot\|$ when there is no ambiguity about the underlying domain.

By applying Green's formulas, the variational formulation for the Poisson problem (\ref{ch2poi}) can be written as:
\begin{eqnarray}\label{ch2poivary}
a(u,v):=\int_\Omega \nabla u\nabla v dx=\int_\Omega fv dx=(f,v),\quad \forall v\in H_0^1(\Omega).
\end{eqnarray}

\noindent The finite element discretized  Poisson problem then reads: find the solution ${u}_n  \in V_n^k$ of the Poisson equation
\be\label{al312poi}
\bal
(\nabla {u}_n, \nabla {v})  = & \langle{f},{v} \rangle \quad \forall {v} \in V_n^k.\\
\eal
\ee 

\begin{figure}[H]
\centering
\includegraphics[width=0.5\textwidth]{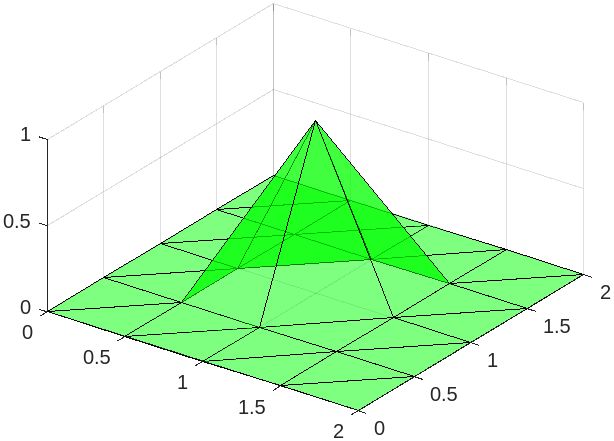}
\caption{Linear hat basis function in 2D}
\end{figure}\label{bas1}

Denote, $\phi$ be the 2D linear Lagrange basis functions as can be seen from the figure (1). Then we can define the basic finite element algorithm as follows: 

\begin{algorithm}[H]
\caption{2D Finite Element Algorithm.}

\begin{algorithmic}
\vspace{0.5cm}

\item[\textbf{Step 1:}]  Create a triangulation : $\mathcal{T}$ of $\Omega \subset R^2$ and define the corresponding space of continuous piece-wise linear functions $V_{h,0}$ with the hat function basis $\{\phi_i\}^{n_i}_{i=1}$.

\item[\textbf{Step 2:}] Generate the $n_i \times n_i$ stiffness matrix $S$ and the $n_i \times 1$ load vector $b$, with entries
        \begin{align*}
            S_{ij}= \int_\Omega \nabla{\phi_j}.\nabla{\phi_i}\;dx,\;\;\;\;b_i=\int_\Omega f\phi_i\;dx.\
        \end{align*}
\item[\textbf{Step 3:}] Solve the linear system of equations
        \begin{align*}
            A\xi=b
        \end{align*}
\item[\textbf{Step 4:}] Write the finite element solution $u_h$ as a linear combination of hat basis functions
        \begin{align*}
            u_h=\sum_{j=1}^{n_i}{\xi_j}\;{\phi_j}
        \end{align*}      
\end{algorithmic}  

\end{algorithm}

\subsection{Well-posedness and Regularity}

\begin{lemma}{(Lax-Milgram)}\label{ch2lax}
Let V be a Hilbert space, let $a(\cdot,\cdot):V\times V\xrightarrow[\textbf{}]{}R $ be a continuous V elliptic bilinear form, and $f:V\xrightarrow[]{}R$ be a continuous linear form. Then the abstract variational problem: Find u such that 
\begin{eqnarray}\label{ch2varypoi}
u \in V \quad {\rm{and}} \quad  v \in V \quad  a(u,v)= f(v) 
\end{eqnarray}
has one and only one solution.
\end{lemma}

For a function $u \in H_0^1(\Omega)$, applying the Poincar\'e-type inequality \cite{Grisvard1}, it follows
$$
a(u,u) = \|\nabla u\|^2 = |u|^2_{H^1(\Omega)} \geq C\|u|^2_{H^1(\Omega)}.
$$

Thus, for any $f \in H^{-1}(\Omega)$, we have by the Lax-Milgram Theorem that Equation (\ref{ch2poivary}) admits a unique solution
$
u \in H_0^1(\Omega).
$

The regularity of the solution $u$ depends on the given data $f$ and the domain geometry \cite{Agmon}, \cite{Blum}. Let $\beta = min_i(\pi/\alpha_i, 1)$ where $\alpha_i$
are interior angles of the polygonal domain $\Omega$. By the regularity theory, the solution $u$ is in $H^{1+\beta}
(\Omega)$. Thus the Poisson Equation 
(\ref{ch2poi}) holds the following regularity estimate 

\begin{eqnarray}\label{poireg}
\|u\|_{H^{1+\beta}(\Omega)} \leq C \|f\|_{H^{-1+\beta}(\Omega)}.
\end{eqnarray}

\subsection{Error Estimates} 
Suppose that the mesh $\maT_n$ consists of quasi-uniform triangles with size $h$. 
The interpolation error estimate on $\maT_n$ (see e.g., \cite{Ciarlet}) for any $v \in H^s(\Omega)$, $s>1$,

\begin{equation}
    \| v - v_I \|_{H^l(\Omega)} \leq Ch^{s-l}\|v\|_{H^s(\Omega)},
\end{equation}\label{errinterpoletion}

where $l= 0, 1$ and $v_I\in V_n^k$ represents the nodal interpolation of $v$.

\begin{lemma}\label{lem22} For a given $ f \in H^{-1}(\Omega)$, let $u$ be the solution of the Poisson problem (\ref{ch2poi}), and ${u}_n$ be the linear finite element approximation (\ref{al312poi}) on a  convex polygonal domain with quasi-uniform meshes. Then it follows

\begin{equation}\label{h1err}
||{u}-{u}_n\|_{[H^1(\Omega)]}  \leq Ch.\\
\end{equation}
\end{lemma}

\begin{proof}: We first derive an important orthogonality result for projections. Let $u$ and $u_h$ be the solution of continuous and discrete equations respectively i.e.
\begin{align*}
    &a(u,v)= \langle f,v \rangle\;\;\;\;\;\;\forall v \in H^{1}_{0}(\Omega),\\ &a(u_h,v)= \langle f,v \rangle\;\;\;\;\;\;\forall v \in V_h.
\end{align*}
Choosing $v \in V_h$ in both equations and subtracting them, we then get an important orthogonality
\begin{equation}
     a(u-u_h,v_h)=0\;\;\;\;\;\;\forall v_h \in V_h,
\end{equation}  
which implies the following optimality of the finite element approximation
\begin{equation}\label{2.2}
     \parallel \nabla(u- u_h) \parallel = \inf_{v_h \in V_h} \parallel \nabla( u- v_h) \parallel 
\end{equation}

Now we replace $v_h$ by the linear nodal interpolation $u_I$ in the equation (\ref{2.2}) which is well defined by the embedding theorems. By (\ref{2.2}), we have
\begin{align*}
   \|{u}-{u}_n\|_{[H^1(\Omega)]} \leq C\parallel \nabla(u- u_h) \parallel \;\leq\; C \parallel \nabla(u- u_I) \parallel \;\lesssim \; Ch \parallel u \parallel_2 \; \lesssim \; C h \parallel f \parallel_{-1} \leq Ch.
\end{align*} 
Here the third inequality is true due to the interpolation error estimate while the fourth inequality is due to the regularity estimate.  
\end{proof}

\begin{lemma}\label{lem23} For a given $ f \in H^{-1}(\Omega)$, let $u$ be the solution of the Poisson problem (\ref{ch2poi}), and ${u}_n$ be the linear finite element approximation (\ref{al312poi}) on a  convex polygonal domain with quasi-uniform meshes. Then it follows
\begin{equation}\label{serrl2}
\|{u}-{u}_n\|_{[L^2(\Omega)]} \leq Ch^2.
\end{equation}
\end{lemma}

Now we estimate $\parallel u-u_h \parallel$. The main technical is the combination of the duality argument and the regularity result. It is known as the Aubin-Nitsche duality argument or simply “Nitsche’s trick”.

\begin{proof}:  By the $H^2$ regularity result, there exist $w \in H^2(\Omega) \cap H^1_0(\Omega)$ such that
\begin{equation}\label{2.6}
    a(w,v)=(u-u_h,v),\;\;\;\;\;\;\;for\; all\; v \in H^1_0(\Omega),
\end{equation}
and $\parallel w \parallel_2 \;\leq C \parallel u-u_h \parallel.$ choosing $v=u-u_h$ in (\ref{2.6}), we get
\begin{align*}
    \parallel u-u_h \parallel^2 &=a(w,u-u_h)\\
    &=a(w-w_I,u-u_h)\\
    &\leq \parallel \nabla(w- w_I) \parallel\; \parallel \nabla(u- u_h) \parallel\;\;\;\;\;\;\;\;(Orthogonality)\\
    &\lesssim h\parallel w\parallel_2 \;\parallel \nabla(u- u_h) \parallel\\
    &\lesssim h\parallel u-u_h \parallel \parallel \nabla(u- u_h) \parallel\;\;\;\;\;\;\;\;\;\;\;\;(regularity).
\end{align*}
Cancelling one $\parallel u-u_h \parallel,$ from both sides we get
\begin{align*}
  \parallel u-u_h \parallel \leq Ch\parallel \nabla(u-u_h)\parallel \lesssim h^2\parallel u \parallel_2.  
\end{align*}

\end{proof} 
For the estimate in $H^1$ norm, when u is smooth enough, we can obtain the optimal first-order estimate. But for $L^2$ norm, the duality argument requires $H^2$ elliptic regularity, which in turn requires that the polygonal domain be convex. In fact, for a non-convex polygonal domain, it will usually not be true that $\parallel u-u_h \parallel = \mathcal{O}(h^2)$ even if the solution $u$ is smooth.

We are interested in the case when $\Omega \subset R^2$ is concave, and thus the solution of Equation (\ref{ch2poi}) possesses corner singularities at vertices of $\Omega$ where some of the interior angles are greater than $\pi$.
It is easy to see that when the maximum angle is larger than $\pi$, i.e., $\Omega$ is concave, $u \not\in H^2(\Omega)$, and thus the finite element approximation based on quasi-uniform grids will not produce the optimal convergence rate. Thus we introduce graded meshes near the singular vertices to recover the optimal convergence rate.

\begin{figure}[H]
\begin{center}
\begin{tikzpicture}[scale=0.2]
%\draw[thick]
\filldraw[thick,color=black!90, fill=green!25]
(-6,-11) -- (2,-11) -- (0,-2) -- (10,-2) -- (8,7) -- (-8,6) -- (-11,-2) -- (-6,-11);
\draw (5,-3) node {$\theta = 0$};
\draw (4,-6) node {$\theta = \omega$};
\draw (7,4) node {$\Omega$};
\draw[thick] (0,-2) node {$\bullet$} node[anchor = north west] {$Q$};
\end{tikzpicture}
\end{center}
\vspace*{-5pt}
    \caption{Domain $\Omega$ containing one re-entrant corner.}
    \label{fig:Omega}
\end{figure}
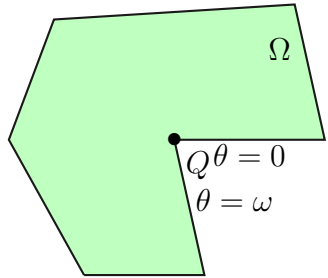

\section{Finite Element Method For Non-Convex Polygonal Domains}

In this section, we shall introduce the weighted Sobolev space $\maK_{\mathbf a}^m(G)$ and provide preliminary results to carry out analysis on graded meshes. On details of weighted Sobolev spaces used here, we refer readers to \cite{Kondratev}, \cite{Bacuta}, \cite{Anna}. Then we use the graded mesh algorithm to improve the convergence rates. To this end, we start with the definition of the weighted Sobolev space.

\subsection{Weighted Sobolev Spaces}

Let, $Q_i$, $i=1,\cdots, N$ are the vertices of domain $\Omega$.
Let $r_i=r_i(x,Q_i)$ be the distance from $x$ to $Q_i$ and let
\begin{eqnarray}\label{eqn.rho}
\rho(x)=\Pi_{1\leq i \leq N} r_i(x,Q_i).
\end{eqnarray}
Let $\mathbf a = (a_1, \cdots,a_i, \cdots, a_N)$ be a vector with $i$th component associated with  $Q_i$. We denote $t+\mathbf a = (t+a_1, \cdots, t+a_N)$, so we have
$$
\rho(x)^{(t+\mathbf a)}=\Pi_{1\leq i \leq N} r_i^{(t+\mathbf a)}(x,Q_i) = \Pi_{1\leq i \leq N} r_i^t(x,Q_i) \Pi_{1\leq i \leq N} r_i^{\mathbf a}(x,Q_i) = \rho(x)^t \rho(x)^{\mathbf{a}}.
$$
Then, we introduce the Kondratiev-type weighted Sobolev spaces for the analysis of the Poisson problem (\ref{e1}).

\begin{definition} \label{wss} (Weighted Sobolev spaces)
For $a\in\mathbb R$, $m\geq 0$, and $G\subset \Omega$,  we define the weighted Sobolev space
$$
\maK_{\mathbf a}^m(G) := \{v|\ \rho^{|\nu|-\mathbf a}\partial^\nu v\in L^2(G), \forall\ |\nu|\leq m \},
$$
where the multi-index $\nu=(\nu_1,\nu_2)\in\mathbb Z^2_{\geq 0}$, $|\nu|=\nu_1+\nu_2$, and $\partial^\nu=\partial_x^{\nu_1}\partial_y^{\nu_2}$.
The $\maK_{\mathbf a}^m(G)$ norm for $v$  is defined by
$$
\|v\|_{\maK_{\mathbf a}^m(G)}=\big(\sum_{|\nu|\leq m}\iint_{G} |\rho^{|\nu|-\mathbf a}\partial^\alpha v|^2dxdy\big)^{\frac{1}{2}}.
$$
\end{definition}

\begin{remark}\label{KHeq}
According to Definition \ref{wss}, in the region that is away from the corners, the weighted space $\maK^m_{\mathbf a}$ is equivalent to the Sobolev space $H^m$. In the neighborhood of $Q_i$, the space $\maK^m_{\mathbf a}(B_i)$ is the equivalent to the Kondratiev space  \cite{Kondratev},\cite{Dauge}, \cite{Grisvard},
$$
\maK_{a_i}^m(B_i) := \{v|\ r_i^{|\nu|-a_i}\partial^\alpha v\in L^2(B_i), \forall\ |\nu|\leq m \},
$$
where $B_i \subset \Omega$ represents the neighborhood of $Q_i$ satisfying $B_i \cap B_j = \emptyset$ for $i\not = j$.
\end{remark}

\subsection{Graded Mesh}

Following \cite{Li3}, \cite{Bacuta}, we now construct a class of suitable graded meshes to obtain the optimal convergence rate of the finite element solution in the presence of the corner singularity in the solution of (\ref{e1}). Starting from an initial triangulation of $\Omega$, we divide each triangle into four
triangles to construct such a sequence of triangulations, which is similar to the regular midpoint refinement. The difference is, in order to attack the corner singularity when we perform
the refinement, we move the middle points of edges towards the singular vertex of $\Omega$. Here a
singular vertex $v_i$  means $Q_i > \pi$.  We now present the construction of graded meshes to improve the convergence rate of the numerical approximation.

\begin{algorithm}
    \caption{Graded Mesh Algorithm}\label{graded}
    \begin{algorithmic}
    
 \item[] Let $\mathcal{T}$ be a triangulation of $\Omega$ with shape-regular triangles. Recall that $Q_i$, $i\;=\;1,...,N$ are the vertices of $\Omega$. Let $AB$ be an edge in the triangulation $\mathcal{T}$ with A and B as the endpoints.Then, in a graded refinement, a new node $D$ on $AB$ is produced according to the following conditions:

        \begin{enumerate}
            \item (Neither $A$ nor $B$ coincides with $Q_i$.) We choose $D$ as the midpoint $(\vert AD \vert =\vert BD \vert)$.
            
       \item ($A$ coincides with $Q_i.)$. We choose $r$ such that $\vert AD \vert = \kappa_{Q_i} \vert AB \vert$, where $\kappa_{Q_i} \in \;(0,0.5)$ is a parameter that will be specified later. See Figure \ref{Fig 2} for example.
        \end{enumerate}

        Then, the graded refinement, denoted by $\kappa(\mathcal{T})$, proceeds as follows. For each triangle $T \in \mathcal{T}$, a new node is generated on each edge of $T$ as described above. Then, $T$ is decomposed into four small triangles by connecting these new nodes. Given an initial mesh $\mathcal{T}_0$ satisfying the condition above, the associated family of graded meshes ${\mathcal{T}_n, n\;\geq\; 0}$ is defined recursively $\mathcal{T}_{n+1}=\kappa(\mathcal{T}_n)$.   
   
  \end{algorithmic}
    
\end{algorithm}

\begin{figure}[h]
\begin{center}
\begin{tikzpicture}[scale=0.3]

%%% top left
%\draw[thick]
\filldraw[color=black!90, fill=green!25]
(-1,1) node[anchor = north] {$x_2$}
-- (-4,7) node[anchor = south] {$x_0$}
-- (-11, 1) node[anchor = north] {$x_1$}
-- (-1,1);

%%% top right
%\draw[thick]
\filldraw[color=black!90, fill=green!25]
(11,1) node[anchor = north] {$x_2$}
-- (8,7) node[anchor = south] {$x_0$}
-- (1, 1) node[anchor = north] {$x_1$}
-- (11,1);

\draw[thick]
(9.5,4)
-- (4.5,4)
-- (6,1) node[anchor = north] {$x_{12}$}
-- (9.5,4);
\draw (3.9,4.3) node {$x_{01}$};
\draw (10,4.3) node {$x_{02}$};

%%% bottom left
%\draw[thick]
\filldraw[color=black!90, fill=green!25]
(-1,-7) node[anchor = north] {$x_2$}
-- (-4,-1) node[anchor = south] {$x_0$}
-- (-11,-7) node[anchor = north] {$x_1$}
-- (-1,-7);
\draw[green,fill=red] (-4,-1) circle (.2);

\draw[thick]
(-23/4, -5/2)
-- (-6,-7+0.04) node[anchor = north] {$x_{12}$}
-- (-13/4, -5/2)
-- (-23/4, -5/2);
\draw (-6.4,-2.2) node {$x_{01}$};
\draw (-2.7,-2.2) node {$x_{02}$};

%%% bottom right
%\draw[thick]
\filldraw[color=black!90, fill=green!25]
(11,-7) node[anchor = north] {$x_2$}
-- (8,-1) node[anchor = south] {$x_0$}
-- (1, -7) node[anchor = north] {$x_1$}
-- (11,-7);
\draw[green,fill=red] (8,-1) circle (.2);

\draw[thick]
(25/4, -5/2)
-- (6,-7+0.04) node[anchor = north] {$x_{12}$}
-- (35/4, -5/2)
-- (25/4, -5/2);
\draw (5.6,-2.2) node {$x_{01}$};
\draw (9.3,-2.2) node {$x_{02}$};
\draw[thick] (29/8, -19/4) -- (7/2,-7) -- (49/8,-19/4) -- (29/8, -19/4);
\draw[thick] (49/8,-19/4) -- (59/8,-19/4) -- (15/2,-5/2) -- (49/8,-19/4);
\draw[thick] (59/8,-19/4) -- (17/2,-7) -- (79/8,-19/4) -- (59/8,-19/4);
\draw[thick] (121/16,-11/8) -- (15/2,-5/2) -- (131/16,-11/8) -- (121/16,-11/8);

%\draw[thick] (11, -7) -- (5,-7);
%\draw[thick] (11, -7) -- (9.75,-8.95) -- (8.75,-7) -- (8.5,-5);
%\draw[thick] (5,-7) -- (5.25,-8.95) -- (7.25,-7) -- (8.5,-5);
%
%\draw[thick] (9.33,-3.67) -- (8.34,-3.67) -- (8.5,-5) -- (9.33,-3.67);
%% \draw[thick] (8.5,-5) -- (9.5,-4);

\end{tikzpicture}
\end{center}
\vspace*{-5pt}
\caption{First row: the initial triangle and the midpoint refinement; second row: graded refinements ($\kappa_{Q_i}<0.5$).}
\label{fig.333}
\end{figure}
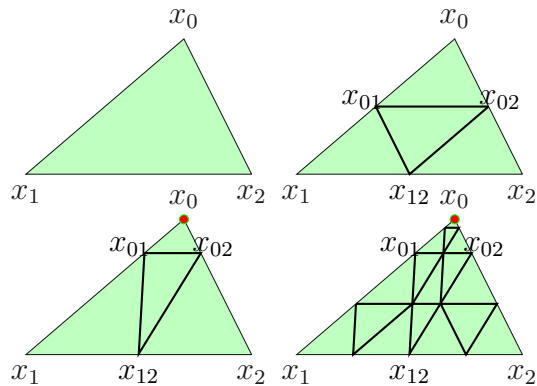

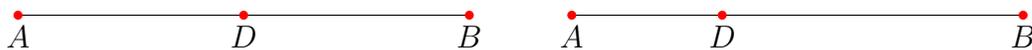
\begin{figure}[h]
\centering
\begin{subfigure}
\centering
\begin{tikzpicture}
\draw (0,0) -- (3,0) -- (6,0);
\filldraw [red] (0,0) circle (1.5pt) ;
\draw (0,0) node[below] {$A$};
\filldraw [red] (3,0) circle (1.5pt) ;
\draw (3,0) node[below] {$D$};
\filldraw [red] (6,0) circle (1.5pt) ;
\draw (6,0) node[below] {$B$};
\end{tikzpicture} 
\end{subfigure}
\hspace{10pt}
\begin{subfigure}
\centering
    \begin{tikzpicture}
\draw (0,0) -- (2,0) -- (6,0);
\filldraw [red] (0,0) circle (1.5pt) ;
\draw (0,0) node[below] {$A$};
\filldraw [red] (2,0) circle (1.5pt) ;
\draw (2,0) node[below] {$D$};
\filldraw [red] (6,0) circle (1.5pt) ;
\draw (6,0) node[below] {$B$}; 
    \end{tikzpicture} 
\end{subfigure} 
\caption{The new node on an edge $AB$. (left): $A \neq Q_i$ and $B \neq Q_i$ (midpoint); 
 (right):  $A=Q_i\; (\vert AB \vert = \kappa_{Q_i} \vert AB \vert ,\; \kappa_{Q_i} < 0.5$).}\label{Fig 2}
\end{figure}

Given a grading parameter $\kappa_{Q_i}$, Algorithm \ref{graded} produces smaller elements near $Q_i$ for better approximation of singular solution. It is an explicit construction of graded meshes based on recursive refinements. See also \cite{Apel}, \cite{Bacuta}, \cite{Hengguang}, \cite{Anna} and references therein for more discussions on the graded mesh.

Note that after $n$ refinements, the number of triangles in the mesh $\maT_n$ is $O(4^n)$, so we denote the mesh size of $\maT_n$ by
\be\label{meshsize}
h = 2^{-n}.
\ee
In Algorithm \ref{graded}, we choose the parameter $\kappa_{Q_i}$ for each vertex $Q_i$ as follows. Given the degree of polynomials $k$, we choose
\be\label{kappainit}
\kappa_{Q_i}=2^{-\frac{\theta}{a_i}}\left(\leq \frac{1}{2}\right),
\ee
where $a_i>0$
%$0<a_i<\alpha_0^i$ 
and $\theta$ could be any possible constants satisfying
\be\label{thetarange}
a_i \leq \theta \leq \min\{k,m\}.
\ee
In (\ref{thetarange}), if we take $a_i=\theta$, the grading parameter $\kappa_{Q_i}=\frac{1}{2}$.

\begin{figure}[ht]
\centering
\subfigure[]{\includegraphics[width=0.32\textwidth]{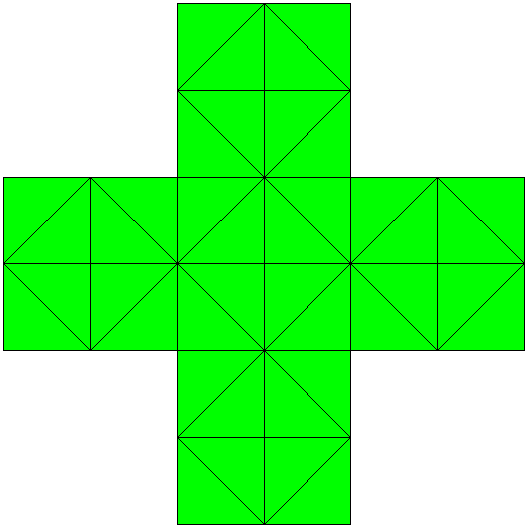}} \hspace{1cm}
\subfigure[]{\includegraphics[width=0.32\textwidth]{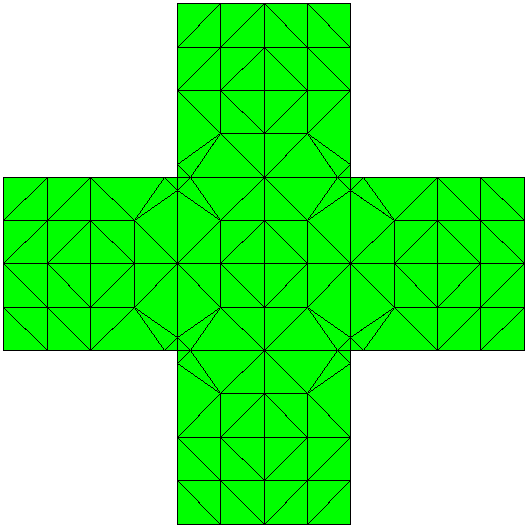}}\\
\subfigure[]{\includegraphics[width=0.32\textwidth]{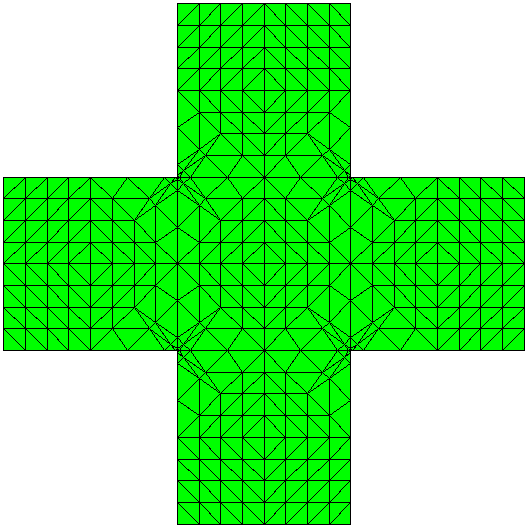}} \hspace{1cm}
\subfigure[]{\includegraphics[width=0.32\textwidth]{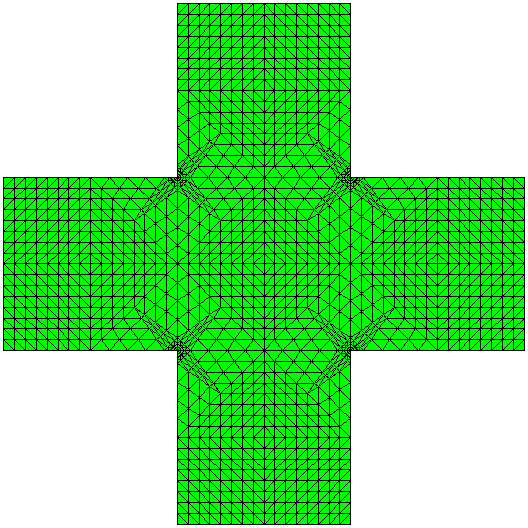}}
\caption{(a) Initial mesh; (b) one refinement; (c) two refinements; (d) three refinements.}\label{comp10}
\end{figure}

Figure \ref{comp10} shows how the graded mesh refinements work on a domain with four re-entrant corners with gradient parameter $\kappa=0.1$ for three consecutive mesh refinements for a given initial mesh 4(a)

\subsection{Error Estimates}
\begin{lem}\label{r1r3}
Let $T_{(0)}\in\mathcal T_{0}$ be an initial triangle of the triangulation $\mathcal T_n$ in Algorithm \ref{graded} with grading parameters $\kappa_{Q_i}$ given by Equation  (\ref{kappainit}). For $m\geq 1, k \geq 1$, we denote $v_I \in V_n^{k}$ the nodal interpolation of $v \in \maK_{\mathbf a+1}^{m+1}(\Omega)$. If $\bar T_{(0)}$ does not contain any vertices $Q_i$, $i=1,\cdots,N$, then
\begin{equation*}
\|v-v_I\|_{H^1(T_{(0)})}\leq Ch^{\min\{k,m\}}
\end{equation*}
where $h = 2^{-n}$.
\end{lem}

\begin{proof}: 
If $\bar T_0$ does not contain any vertices $Q_i$ of the domain $\Omega$, we have $v\in \maK_{\mathbf a+1}^{m+1}(\Omega) \subset H^{m+1}(T_{(0)})$ (see Remark \ref{KHeq}) and the mesh on $T_{(0)}$ is quasi-uniform (Algorithm \ref{graded}) with size $O(2^{-n})$. Therefore, based on the standard interpolation error estimate, we have
\begin{eqnarray}\label{1.1}
\|v-v_I\|_{H^1(T_{(0)})}\leq Ch^{\min\{k,m\}}\|v\|_{H^{m+1}(T_{(0)})}.
\end{eqnarray}

\end{proof}

We now study the interpolation error in the neighborhood $Q_i$, $i=1,\cdots, N$.
In the rest of this subsection, we assume $T_{(0)}\in\mathcal T_0$ is an initial triangle such that the $i$th vertex $Q_i$ is a vertex of $T_{(0)}$.
We first define mesh layers on $T_{(0)}$ which are collections of triangles in $\mathcal T_n$.

\begin{definition} (Mesh layers) Let $T_{(t)}\subset T_{(0)}$ be the triangle in $\mathcal T_t$, $0\leq t\leq n$, that is attached to the singular vertex $Q_i$ of $T_{(0)}$. For $0\leq t<n$, we define the $t$th mesh layer of $\mathcal T_n$ on $T_{(0)}$ to be the region $L_{t}:=T_{(t)}\setminus T_{(t+1)}$; and for $t=n$, the $n$th layer is $L_{n}:=T_{(n)}$.  See Figure \ref{fig.layer} for example.
\end{definition}

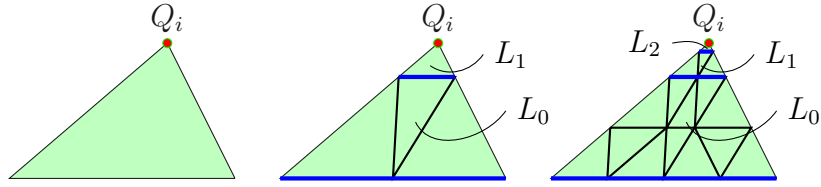
\begin{figure}[ht]
\begin{center}
\begin{tikzpicture}[scale=0.30]
%%% left
%\draw[thick]
\filldraw[color=black!90, fill=green!25]
(-13,-7)
-- (-16,-1) node[anchor = south] {$Q_i$}
-- (-23, -7)
-- (-13,-7);
\draw[green,fill=red] (-16,-1) circle (.2);

%%% middle
%\draw[thick]
\filldraw[color=black!90, fill=green!25]
(-1,-7)
-- (-4,-1) node[anchor = south] {$Q_i$}
-- (-11,-7)
-- (-1,-7);
\draw[green,fill=red] (-4,-1) circle (.2);

\draw[thick]
(-23/4, -5/2)
-- (-6,-7+0.04)
-- (-13/4, -5/2)
-- (-23/4, -5/2);

\draw[blue, ultra thick] (-13/4, -5/2) -- (-23/4, -5/2);
\draw[blue, ultra thick] (-11,-7) -- (-1,-7);
\draw (-5,-4.5) to[out=-65,in=235] (-1,-4) node[anchor = west] {$L_0$};
\draw (-4.35,-2) to[out=-35,in=235] (-2,-1.5) node[anchor = west] {$L_1$};
%%% right
%\draw[thick]
\filldraw[color=black!90, fill=green!25]
(11,-7)
-- (8,-1) node[anchor = south] {$Q_i$}
-- (1, -7)
-- (11,-7);
\draw[green,fill=red] (8,-1) circle (.2);

\draw[thick]
(25/4, -5/2)
-- (6,-7+0.04)
-- (35/4, -5/2)
-- (25/4, -5/2);

\draw[thick] (29/8, -19/4) -- (7/2,-7) -- (49/8,-19/4) -- (29/8, -19/4);
\draw[thick] (49/8,-19/4) -- (59/8,-19/4) -- (15/2,-5/2) -- (49/8,-19/4);
\draw[thick] (59/8,-19/4) -- (17/2,-7) -- (79/8,-19/4) -- (59/8,-19/4);
\draw[thick] (121/16,-11/8) -- (15/2,-5/2) -- (131/16,-11/8) -- (121/16,-11/8);

\draw[blue, ultra thick] (131/16,-11/8) -- (121/16,-11/8);
\draw[blue, ultra thick] (-13/4+12, -5/2) -- (-23/4+12, -5/2);
\draw[blue, ultra thick] (-11+12,-7) -- (-1+12,-7);
\draw (-5+12,-4.5) to[out=-65,in=235] (-1+12,-4) node[anchor = west] {$L_0$};
\draw (-4.35+12,-2) to[out=-35,in=235] (-2+12,-1.5) node[anchor = west] {$L_1$};
\draw (127/16,-1.25) to[out=500,in=10] (102/16,-1) node[anchor = east] {$L_2$};
\end{tikzpicture}
\end{center}
\vspace*{-5pt}
\caption{The initial triangle $T_{(0)}$ with singular vertex $Q_i$ and mesh layers.}
\label{fig.layer}
\end{figure}

\begin{remark}
The triangles in $\mathcal T_n$ constitute $n$ mesh layers on $T_{(0)}$. According to Algorithm \ref{graded} and the choice of grading parameters $\kappa_{Q_i}$ given by Equation  (\ref{kappainit}), the mesh size in the $t$th  layer $L_t$ is
\begin{equation}\label{eqn.size}O(\kappa_{Q_i}^t2^{t-n}). \end{equation}
Meanwhile, the weight function $\rho$ in Equation  (\ref{eqn.rho}) satisfies
\begin{eqnarray}\label{eqn.dist}
\rho=O(\kappa_{Q_i}^t) \ \ \ {\rm{in\ }} L_t\ (0\leq t< n) \qquad {\rm{and}}  \qquad \rho \leq C\kappa_{Q_i}^n \ \ \ {\rm{in\ }} L_n.
\end{eqnarray}
\end{remark}

Although the mesh size varies in different layers, the triangles in $\mathcal T_n$ are shape regular. In addition, using the local Cartesian coordinates such that $Q$ is the origin,  the mapping
\begin{eqnarray}\label{eqn.map}
\mathbf B_{t}= \begin{pmatrix}
  \kappa_{Q_i}^{-t}   &   0 \\
  0    &   \kappa_{Q_i}^{-t} \\
\end{pmatrix},\qquad 0\leq t\leq n
\end{eqnarray}
is a bijection between $L_t$ and $L_0$ for $0\leq t<n$ and a bijection between $L_n$ and $T_{(0)}$. We call $L_0$ (resp. $T_{(0)}$) the reference region associated to $L_t$ for $0\leq t<n$ (resp. $L_n$).

With the mapping (\ref{eqn.map}), we have that for any point $(x, y)\in L_t$, $0\leq t\leq n$, the image point $(\hat x, \hat y):=\mathbf B_t(x,y)$ is in its reference region. We then introduce the following result from \cite[Lemma 4.5]{Peimeng}.

\begin{remark}\label{dilation}
For $0\leq t\leq n$,
given a function $v(x, y) \in \maK_{a}^{l}(L_t)$, the function $\hat v(\hat x, \hat y):=v(x, y)$ belongs to $\maK_{a}^{l}(\hat L)$, where $(\hat x, \hat y):=\mathbf B_t(x,y)$, $\hat L=L_0$ for $0\leq t< n$, and $\hat L=T_{(0)}$ for $t=n$. Then, it follows
\begin{equation*}
\|\hat v(\hat x, \hat y)\|_{\maK_{a}^{l}(\hat L)} = \kappa_{Q_i}^{t(a-1)} \|v(x,y)\|_{\maK_{a}^{l}(L_i)}.
\end{equation*}
\end{remark}

We then derive the interpolation error estimate in each layer.

\begin{lem}\label{TNtri}
For $k \geq 1, m \geq 1$, set $\kappa_{Q_i}$ in Equation (\ref{kappainit}) with $\theta$ satisfying (\ref{thetarange}) for the graded mesh on $T_{(0)}$.
Let $h:=2^{-n}$, then in the $t$th layer $L_t$ on $T_{(0)}$, $0\leq t<n$, if $v_I \in V_n^k$ be the nodal interpolation of $v\in \maK_{\mathbf a+1}^{m+1}(\Omega)$, it follows
\be\label{projgh1}
|v-v_{I}|_{H^1(L_t)} \leq Ch^{\theta}\|v\|_{\maK_{a_i+1}^{m+1}(L_t)}
\ee

\end{lem}
\begin{proof}: 
For $L_t$ associated with $Q_i$, $0\leq t<n$, the space $\maK_{a_i+1}^{m+1}(L_t)$ is equivalent to $H^{m+1}(L_t)$. Therefore, $v$ is a continuous function in $L_t$.
For any point $(x, y)\in L_t$, let $(\hat x, \hat y)=\mathbf B_t(x,y)\in L_0$. For $v(x, y)$ in $L_t$, we define $\hat v(\hat x, \hat y):=v(x, y)$ in $L_0$.

Using the standard interpolation error estimate, the scaling argument,  the estimate in (\ref{eqn.size}), and the mapping in (\ref{eqn.map}), we have
\begin{eqnarray*}
 |v-v_I|_{H^1(L_t)}&=& |\hat v-\hat v_I|_{H^1(L_0)}\leq C 2^{(t-n)\mu}\|\hat v\|_{\maK_{a_i+1}^{m+1}(L_0)}
 \leq C 2^{(t-n)\mu}\kappa_{Q_i}^{a_it}\|v\|_{\maK_{a_i+1}^{m+1}(L_t)},
\end{eqnarray*}
where we have used Lemma \ref{dilation} in the last inequality.
Since $\kappa_{Q_i} = 2^{-\frac{\theta}{a_i}}$,
so we have
$\kappa_{Q_i}^{a_it} = 2^{-\theta t}$.
Set $\mu=\min\{k,m\}$, by $\theta \leq \mu$ from (\ref{thetarange}) and $t<n$, we have $2^{(n-t)(\theta-\mu)}<2^0=1$.
Therefore, we have the estimate
\begin{eqnarray*}
 |v-v_I|_{H^1(L_t)}
 &\leq& C 2^{(t-n)\mu-\theta t}\|v\|_{\maK_{a_i+1}^{m+1}(L_t)} = C2^{-n\theta} 2^{(n-t)(\theta-\mu)}\|v\|_{\maK_{a_i+1}^{m+1}(L_t)}\\
 &\leq& C2^{-n\theta} \|v\|_{\maK_{a_i+1}^{m+1}(L_t)} \leq Ch^{\theta} \|v\|_{\maK_{a_i+1}^{m+1}(L_t)}.
\end{eqnarray*}

\end{proof}

Before deriving the interpolation error estimate in the last layer $L_n$ on $T_{(0)}$, we first introduce the following results.

\begin{remark}\label{Lnrela}
For $\forall v \in \maK_a^l(L_n)$, if $0\leq l'\leq l$ and $a'\leq a$, then it follows
\be
\|v\|_{\maK_{a'}^{l'}(L_n)}\leq C \kappa_{Q_i}^{n(a-a')} \|v\|_{\maK_{a}^{l}(L_n)}.
\ee
\end{remark}

\begin{remark}
    
\label{HtoKbdd}
For $\forall v \in \maK_a^l(L_n)$ , if $a\geq l$, then it follows that
\be\label{HtoKbddfor}
\|v\|_{H^l(L_n)} \leq C \kappa_{Q_i}^{n(a-l)} \|v\|_{\maK_a^l(L^n)}.
\ee
\end{remark}

\begin{lem}\label{TNtri2}
For $k \geq 1, m \geq 1$, set $\kappa_{Q_i}$ in (\ref{kappainit}) with $\theta$ satisfying (\ref{thetarange}) for the graded mesh on $T_{(0)}$.
Let $h:=2^{-n}$, then in the $n$th layer $L_n$ on $T_{(0)}$ for $n$ sufficiently large, if $v_I \in V_n^k$ be the nodal interpolation of $v\in \maK_{\mathbf a+1}^{m+1}(\Omega)$, it follows
\be\label{projgh12}
|v-v_{I}|_{H^1(L_n)} \leq Ch^{\theta}\|v\|_{\maK^{m+1}_{{a_i}+1}(L_{n})}
\ee

\end{lem}
\begin{proof}: 
Recall the mapping $\mathbf B_n$ in (\ref{eqn.map}). For any point $(x, y)\in L_n$, let $(\hat x, \hat y)=\mathbf B_n(x,y)\in T_{(0)}$.

Let $\eta: T_{(0)} \rightarrow [0, 1]$ be a  smooth function that
is equal to $0$ in a neighborhood of $Q_i$, but is equal
to 1 at all the other nodal points in $\mathcal T_0$.
For a function $v(x, y)$ in $L_n$, we define $\hat v(\hat x, \hat y):=v(x, y)$ in $T_{(0)}$. We take $w=\eta \hat v$ in $T_{(0)}$. Consequently, we have for $l\geq 0$
\begin{equation}\label{eqn.aux111}
\bal
\|w\|^2_{\maK^{l}_{1}(T_{(0)})} & = & \|\eta
\hat v\|^2_{\maK^{l}_{1}(T_{(0)})} \leq C
\|\hat v\|^2_{\maK^{l}_{1}(T_{(0)})},
\eal
\end{equation}

where $C$ depends on $l$ and the smooth function $\eta$. Moreover, the condition $\hat v\in \maK_{a^i+1}^{m+1}(T_{(0)})$ with and $m\geq 2$ implies $\hat v(Q)=0$ (see, e.g., \cite[Lemma 4.7]{Hengguang}).
Let $w_{\hat I}$ be the nodal interpolation of $w$ associated with the mesh $\mathcal T_0$ on $T_{(0)}$.
Therefore, by the definition of $w$, we have
\begin{eqnarray}\label{wi}
w_{\hat I}=\hat v_{\hat I} = \widehat{v_{I}} \quad {\rm{in}}\ T_{(0)}.
\end{eqnarray}

Note that the $\maK^{l}_{1}$ norm and the $H^l$ norm are equivalent for $w$ on $T_{(0)}$, since $w=0$ in the neighborhood of the vertex $Q_i$. Let $r$ be the distance from $(x,y)$ to $Q_i$, and $\hat r$ be the distance from $(\hat x,\hat y)$ to $Q_i$. Then, by the definition of the weighted space, the scaling argument, Equations (\ref{eqn.aux111}),  (\ref{wi}),  and (\ref{eqn.dist}), we have
\begin{align*}
    |v-v_{I}|_{H^1(L_{n})}^2 &\leq C\|v-v_{I}\|_{\maK^1_{{1}}(L_{n})}^2 \\
    &\leq C\sum_{|\nu|\leq1}\|r(x,y)^{|\nu|-1}\partial^\nu (v- v_I)\|_{L^2(L_{n})}^2\\
    &\leq C\sum_{|\nu|\leq 1}\|\hat r(\hat{x},\hat{y})^{|\nu|-1}\partial^\nu (\hat v- \widehat{v_I})\|_{L^2(T_{(0)})}^2\leq C\|\hat v- w+w-\widehat{v_I}\|_{\maK^1_{1}( T_{(0)})}^2\\
    &\leq C\big( \|\hat v-w\|^2_{\maK^1_{1}(T_{(0)})} +
\|w-\widehat{v_I}\|^2_{\maK^1_{1}( T_{(0)}  )}\big) = C\big( \|\hat v-w\|^2_{\maK^1_{1}(T_{(0)})} +
\|w-w_{\hat I}\|^2_{\maK^1_{1}( T_{(0)} )}\big)\\
&\leq  C\big(   \|\hat v\|^2_{\maK^1_{1}(T_{(0)} )} +
\|w\|^2_{\maK^{m+1}_{1}( T_{(0)}  )}\big) \\
&\leq C\big(   \|\hat v\|^2_{\maK^1_{1}(T_{(0)})} +
\|\hat v\|^2_{\maK^{m+1}_{1}( T_{(0)} )}\big) =  C\big(   \|v\|^2_{\maK^1_{1}(L_n)} +
\|v\|^2_{\maK^{m+1}_{1}( L_n)}\big)\\
&\leq C
\kappa_{Q_i}^{2na_i}\|v\|_{\maK^{m+1}_{{a_i}+1}(L_{n})}^2\\
& \leq C
2^{-2n\theta}\|v\|_{\maK^{m+1}_{{a_i}+1}(L_{n})}^2\\
&\leq C
h^{2\theta}\|v\|_{\maK^{m+1}_{{a_i}+1}(L_{n})}^2,
\end{align*}

where the ninth and tenth relationships are based on Remark \ref{dilation} and Remark \ref{Lnrela}, respectively.
This completes the proof of (\ref{projgh1}).

\end{proof}
\begin{lemma}\label{gradprojerr}
\cite{PYin} Let $\mathcal T_{0}$ be an initial triangle of the triangulation $\mathcal T_n$ in Algorithm \ref{graded} with grading parameters $\kappa_{Q_i}$ in (\ref{kappainit}).
For $k\geq1, m \geq 1$,
if $v_I \in V_n^k$  be the nodal interpolation of $v\in \maK_{\mathbf a+1}^{m+1}(\Omega)$ . Then, it follows the following interpolation error
\be\label{projgherr}
\|v-v_{I}\|_{H^1(\Omega)} \leq Ch^{\theta} \|v\|_{\maK_{\mathbf a+1}^{m+1}(\Omega)}
\ee
where $h:=2^{-n}$, and $\theta$ satisfying (\ref{thetarange}).
\end{lemma}
\begin{proof}: 
By summing the estimates in Lemmas \ref{r1r3}, \ref{TNtri}, and \ref{TNtri2}, we have
\bes
\bal
\|v-v_{I}\|^2_{H^1(\Omega)} =& \sum_{T_{(0)} \in \mathcal T_{0}} \|v-v_{I}\|^2_{H^1(T_{(0)})} \leq Ch^{2\theta}\|v\|^2_{\maK_{\mathbf a+1}^{m+1}(\Omega)}
\eal
\ees
\end{proof}
Recall that the threshold of grading parameter $\kappa_{Q_i}$ in obtaining the optimal convergence rates, we always assume $1\leq k\leq m$ in the following discussions, otherwise we just replace $k$ by $\min\{k,m\}$.
In this section, we assume that $f\in \maK_{\mathbf{a}-1}^{m-1}(\Omega)$ with $0< \mathbf a < \bm\beta_0$,
where $\bm{\beta}_0 = (\frac{\pi}{\omega_1}, \cdots, \frac{\pi}{\omega_N})$. The regularity estimate \cite{Bacuta} for the Poisson problem (\ref{e1}) on weighted Sobolev space, follows that 
\be\label{wreggrad}
\|u\|_{\maK_{\mathbf{b}+1}^{m+1}(\Omega)} \leq C\|f\|_{\maK_{\mathbf{b}-1}^{m-1}(\Omega)},
\ee

Since the bilinear functional of the Poisson equation (\ref{e1}) is coercive and continuous on $V_n^k$, so we have by C\'ea's Theorem, 
\be\label{ceathmgrade}
\|u-u_n\|_{H^1(\Omega)} \leq C \inf_{v \in V_n^k} \|u-v\|_{H^1(\Omega)}.
\ee
Recall that $\beta_0 = \min_i\{\beta_0^i\}=\frac{\pi}{\omega}$ are the thresholds corresponding to the largest interior angle $\omega$,
then we have the following result.

\begin{theorem}\label{h1error}
\cite{dilhara} \cite{DilDiss}  Set the grading parameters $\kappa_{Q_i}=2^{-\frac{\theta}{a_i}}$ with $0<a_i<\beta_0^i$, $\theta$ being any constant satisfying $a_i\leq \theta \leq k$, and $\theta'=\min\left\{\max\{\theta, \beta_0\}, k\right\}$ satisfying $a_i \leq \theta' \leq k$.
Let $u_n\in V_{n}^{k}$ be the solution of finite element solution of Equation (\ref{al312poi}), and $u$ is the solution of the Poisson problem (\ref{e1}), then it follows
\be\label{phiH1errg3.1}
\|u-u_n\|_{H^1(\Omega)} \leq Ch^{\theta'}
\ee
where $h:=2^{-n}$.
\end{theorem}
\begin{proof}: 
By Equation (\ref{ceathmgrade}) and the interpolation error estimates in Lemma \ref{gradprojerr} under the regularity result in Equation  (\ref{wreggrad}) and $\kappa_{Q_i}=2^{-\frac{\theta}{a_i}}$,
we have the estimate  
$$
\|u-u_n\|_{H^1(\Omega)} \leq C \|u-u_I\|_{H^1(\Omega)} \leq Ch^{\theta'}.
$$

\end{proof}

\begin{theorem}\label{l2erro}
\cite{dilhara} Set the grading parameters $\kappa_{Q_i}=2^{-\frac{\theta}{a_i}}$ with $0<a_i<\beta_0^i$, $\theta$ being any constant satisfying $a_i\leq \theta \leq k$, and $\theta'=\min\left\{\max\{\theta, \beta_0\}, k\right\}$ satisfying $a_i \leq \theta' \leq k$.
Let $u_n\in V_{n}^{k}$ be the solution of finite element solution of Equation (\ref{al312poi}), and $u$ is the solution of the Poisson problem (\ref{e1}), then it follows
\be\label{phiH1errg3.1.1}
 \|u-u_n\|\leq Ch^{\min\left\{2\theta', \theta'+1\right\}},
\ee
where $h:=2^{-n}$.
\end{theorem}
\begin{proof}:  Consider the Poisson problem 
\be\label{dualVL2w}
-\Delta v = u - u_n \text{ in } \Omega, \quad v =0 \text{ on } \partial \Omega.
\ee
Then we have
\be\label{dualL2w}
\|u - u_n\|^2=(\nabla (u - u_n), \nabla v).
\ee
Subtract Equation (\ref{al312poi}) from weak formulation of Equation (\ref{e1}), we have the Galerkin orthogonality,
\be\label{poigo}
(\nabla (u - u_n), \nabla \phi) = 0, \quad \forall \phi \in V_n^k.
\ee
Setting $\phi=v_I\in V_{n}^{k}$ the nodal interpolation of $v$ and subtract Equation  (\ref{poigo}) from  Equation (\ref{dualL2w}), we have
\be\label{dualL2w2}
\bal
\|u - u_n\|^2=(\nabla (u - u_n), \nabla (v-v_I)) \leq \|u-u_n\|_{H^1(\Omega)} \|v-v_I\|_{H^1(\Omega)}.
\eal
\ee
Similarly, the solution
$v\in \mathcal K^{2}_{\mathbf b'+1}(\Omega)$ satisfies the regularity estimate
\be\label{poiregw1}
\|v\|_{K^{2}_{\mathbf a'+1}(\Omega)} \leq C\|u-u_n\|_{K^{0}_{\mathbf a'-1}(\Omega)} \leq C\|u-u_n\|,
\ee
where the $i$th entry of $\mathbf{a}'$ satisfying $a'_i=\min\left\{a_i,1\right\}$.
By Lemma \ref{gradprojerr} with grading parameter $\kappa_{Q_i} =2^{-\frac{\theta'}{a_i}}$ again,
we have the interpolation error
\be\label{dualintererr3.1}
\| v - v_I \|_{H^1(\Omega)} \leq Ch^{\min\{\theta', 1\}}\|v\|_{K^{2}_{\mathbf b'+1}(\Omega)}.
\ee
The $L^2$ error estimate in Equation  (\ref{phiH1errg3.1}) can be obtained by combining Equations (\ref{dualL2w2}), (\ref{poiregw1}), and (\ref{dualintererr3.1}).
\end{proof}

\section{Numerical Results}\label{sec3}
In this section, we present numerical tests to validate our theoretical predictions for the proposed finite element algorithm solving the Poisson problem under uniform and graded meshes. If an exact solution (or vector) $v$ is unknown, we use the following numerical convergence rate

\begin{eqnarray}\label{rate}
{\mathcal R}=\log_2\frac{|v_j-v_{j-1}|_{[H^l(\Omega)]}}{|v_{j+1}-v_j|_{[H^l(\Omega)]}},
\end{eqnarray}

$l=0,1$ as an indicator of the actual convergence rate.
Here $v_j$ denotes the finite element solution on the mesh $\mathcal T_j$ obtained after $j$  refinements of the initial triangulation $\mathcal T_0$.
All the numerical examples are tested on MATLAB R2022a in MacBook Air (M1, 2020) with 8 GB memory by adapting iFEM MATLAB package \cite{Chen}.

\begin{example}
In this example, we solve the Poisson equation (\ref{e1}) using linear finite elements. We consider a convex polygonal domain as illustrated in figure (\ref{conve}a) and apply a Dirichlet boundary condition $u = 0$ on $\partial \Omega$, with $f=2$. As we increase the number of uniform mesh refinements, both the $H^1$ and $L^2$ errors gradually decrease. We have numerically obtained the $H^1$ convergent rate ${\mathcal R}=0.9941$ and the $L^2$ convergent rate ${\mathcal R}=1.9919$, which are very close to the theoretical convergent rates ${\mathcal R}=1$ and ${\mathcal R}=2$, respectively, as expected based on Lemma \ref{lem22} and Lemma \ref{lem23}. Figures (\ref{conve}b) and (\ref{conve}c) depict two consecutive uniform mesh refinements starting from the initial mesh shown in figure (\ref{conve}a). Finally, figure (\ref{sol7}) displays the numerical solution after seven mesh refinements, observed from two different view angles.

\begin{figure}[ht]
\centering
\subfigure[]{\includegraphics[width=0.24\textwidth]{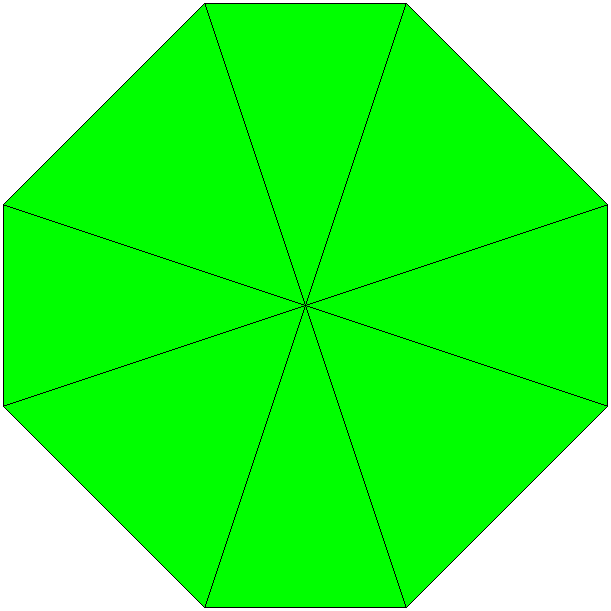}} \hspace{0.4cm}
\subfigure[]{\includegraphics[width=0.24\textwidth]{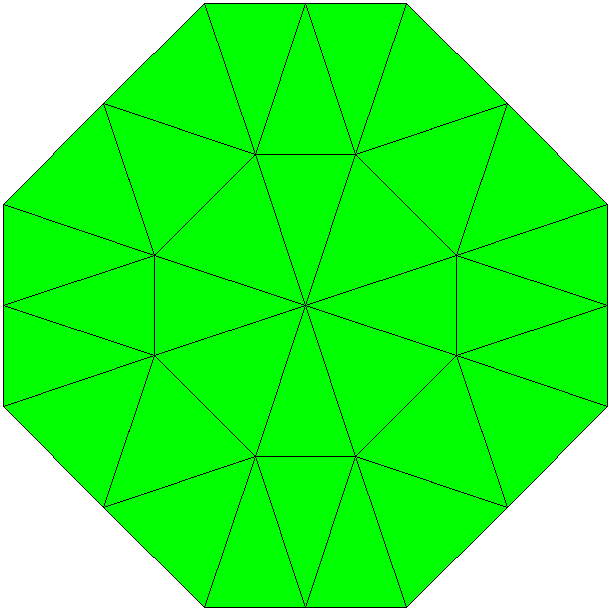}} \hspace{0.4cm}
\subfigure[]{\includegraphics[width=0.24\textwidth]{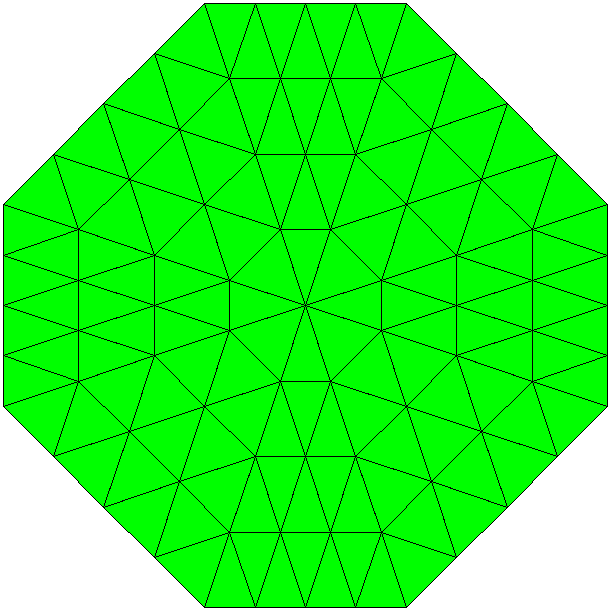}}
\caption{(a) Initial mesh; (b) First mesh refinement; (c) Second mesh refinement.}\label{conve}
\end{figure}

\begin{table}[!t] %\tabcolsep0.04in
\centering
\def\arraystretch{1.20}
\caption{Errors and convergent rates under octagon domain on quasi-uniform meshes.}
\tabcolsep=22pt%%
\vspace{-0.5\baselineskip}
%\resizebox{0.9\textwidth}{!}{
\begin{tabular}[c]{|c|c|c|c|c|}
\hline
$j$& \multicolumn{1}{|c |}{$H^{1}$ error}& \multicolumn{1}{|c |}{$H^{1}$ rate} & \multicolumn{1}{| c|}{$L^{2}$ error} & \multicolumn{1}{| c|}{$L^{2}$ rate}\\

\hline
2  & 2.9515 & - & 1.4257 & - \\
\cline{1-5}
\hline
3  & 2.4415 & 0.2737 &  0.7706 & 0.8876  \\
\cline{1-5}
\hline
4  & 1.4404 & 0.7613 & 0.2491 & 1.6291   \\
\cline{1-5}
\hline
5  &   0.7853 & 0.8751 & 0.0714 & 1.8027  \\
\cline{1-5}
\hline
6  & 0.4123 &  0.9294 & 0.0192& 1.8917 \\
\cline{1-5}
\hline
7  & 0.2122 & 0.9585 & 0.0050 & 1.9381 \\
\cline{1-5}
\hline
8  & 0.1080 &0.9749 &  0.0013  & 1.9636\\
\cline{1-5}
\hline
9  & 0.0546 & 0.9846 & 3.2681e-04 & 1.9782 \\
\cline{1-5}
\hline
10 & 0.0275 & 0.9905 & 8.2455e-05 & 1.9868  \\
\cline{1-5}
\hline
11 & 0.0138 & 0.9941 & 2.0729e-05 & 1.9919 \\
\cline{1-5}
\hline

\end{tabular}\label{rate2}
%}
\end{table}
  
\end{example}

\begin{figure}[H]
\centering
\subfigure[]{\includegraphics[width=0.43\textwidth]{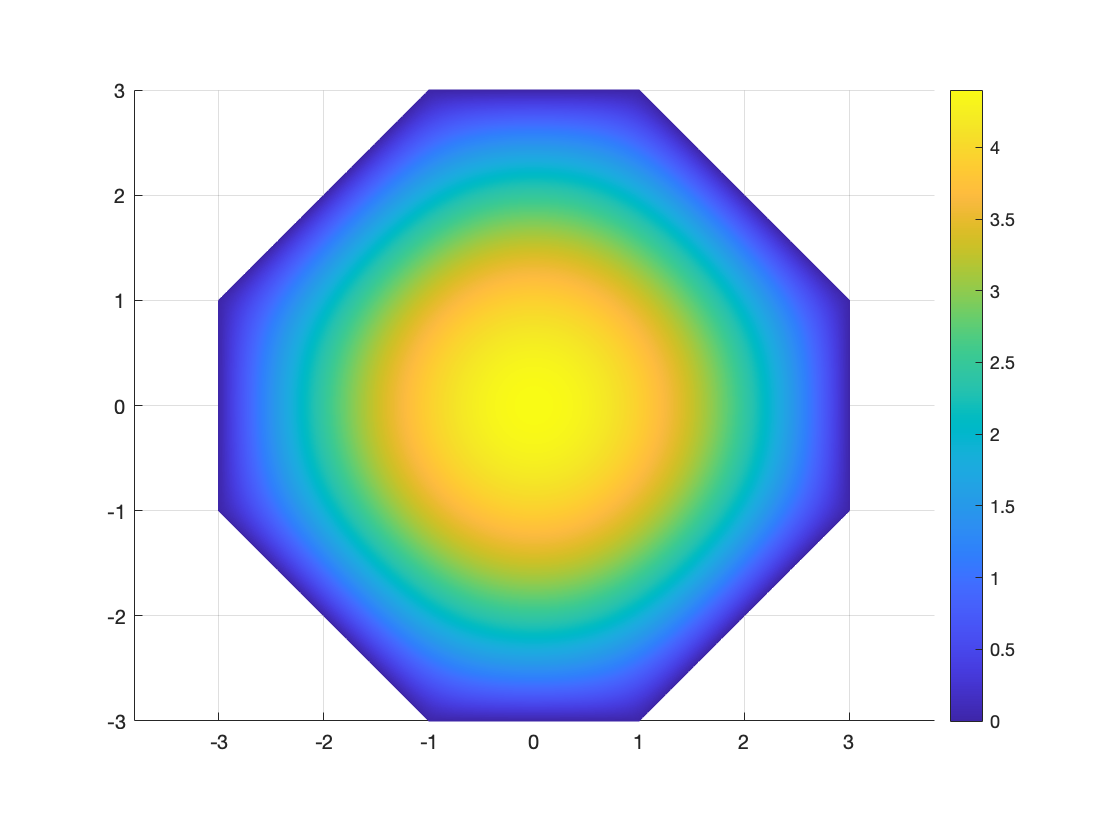}}
\subfigure[]{\includegraphics[width=0.43\textwidth]{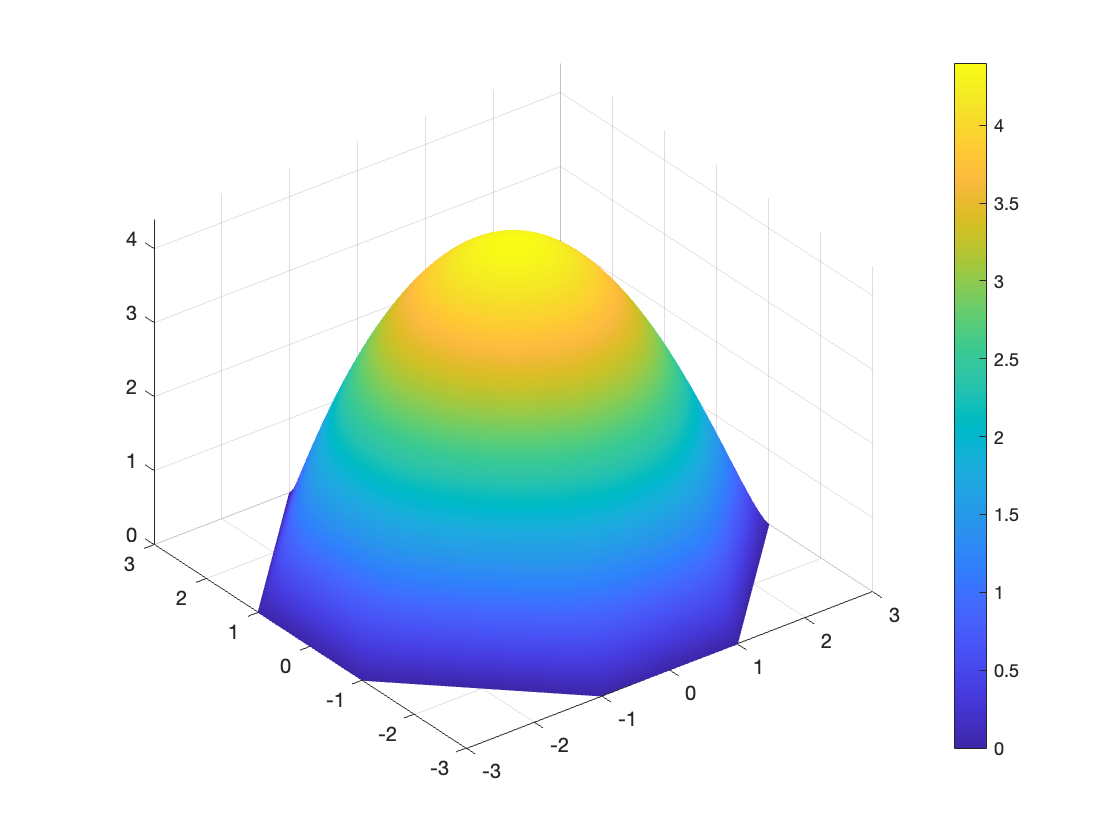}}

\caption{Numerical solution after 7 uniform mesh refinements.}\label{sol7}
\end{figure}

\begin{example}
In this example, we solve the Poisson equation on a non-convex domain (see figure \ref{plane}a)  with seven re-entrant corners. with $f=\frac{1}{2}$\; for a sequence of grading parameters $\kappa = 0.1, 0.2, 0.3, 0.4, 0.5 $ where $\kappa =0.5 $ is the uniform mesh refinements. In the presence of re-entrant corners uniform mesh refinements (i.e. $\kappa =0.5 $ ) won't be able to capture the singular behavior of the solution. Thus as you can see from Tables 2 and 3, after 10 mesh refinements $L^2$ convergent rate is $1.2989$ and the $H^1$ convergent rate is $0.6842$ which is not the optimal convergent rate. However, with the graded mesh refinements we were able to obtain the optimal convergent rate as you can see from tables 3 and 4. For examples, in table 2, numerical $L^2$ convergent rate ${\mathcal R}= 1.9868$ for $\kappa=0.1$ after 10 mesh refinements. This is in strong agreement with the Theorem \ref{h1error} where the theoretical $L^2$  convergent rate is ${\mathcal R}= 2$ under  $L^2$ norm. Moreover,  in table 4, numerical $H^{1}$  convergent rate is ${\mathcal R}= 0.9943$ for $\kappa=0.1$ after 10 mesh refinements. This is also in strong agreement with the Theorem \ref{l2erro} where the theoretical $H^{1}$ convergent rate is ${\mathcal R}= 1$. Finally, figure (\ref{plan3d}) displays the numerical solution after seven mesh refinements, observed from two different view angles.

\end{example}
\begin{figure}[ht]
\centering
\subfigure[]{\includegraphics[width=0.30\textwidth]{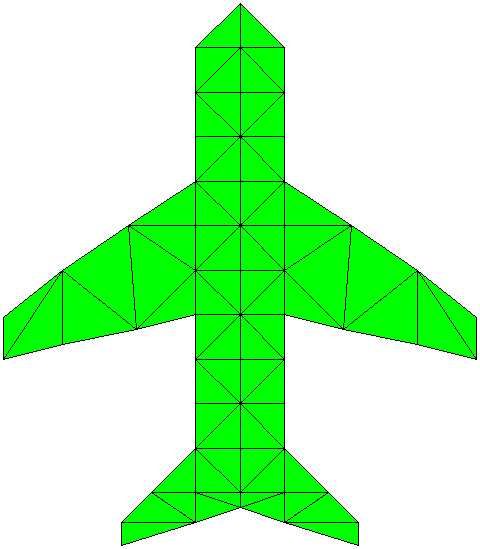}} \hspace{0.5cm}
\subfigure[]{\includegraphics[width=0.30\textwidth]{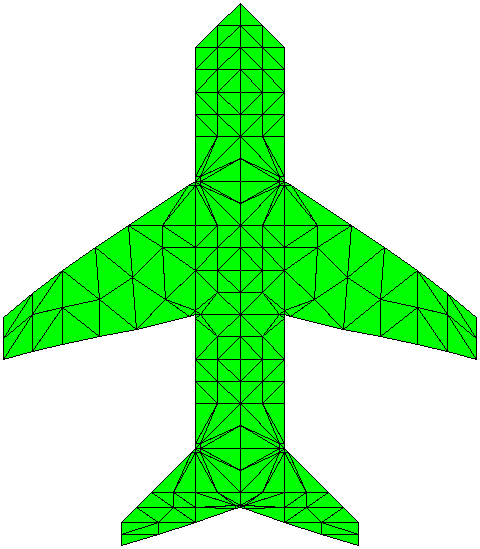}}  \hspace{0.5cm}
\subfigure[]{\includegraphics[width=0.30\textwidth]{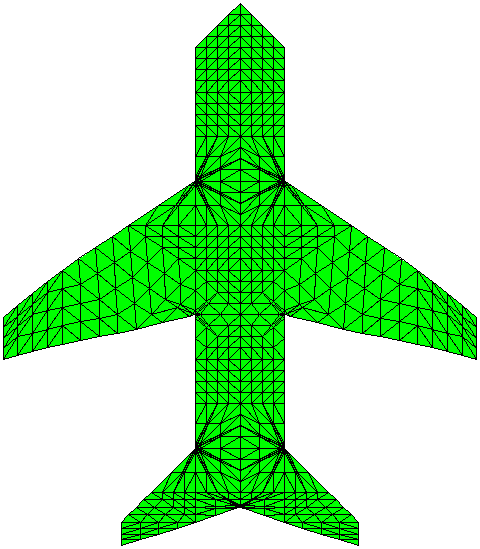}}
\caption{Initial mesh (a) with two consecutive graded mesh refinements (b) and (c) for $\kappa=0.1$.}\label{plane}
\end{figure}

\begin{table}[ht] %\tabcolsep0.04in
\centering
\def\arraystretch{1.20}
\caption{$L^2$ convergent rates for different gradient parameters $\kappa$ for consecutive mesh levels $j$.}
\tabcolsep=18pt%%
\vspace{-0.5\baselineskip}
\begin{tabular}[c]{|c|c|c|c|c|c|}
\hline

$j $& \multicolumn{1}{|c |}{$\kappa=0.1$}& \multicolumn{1}{|c |}{$\kappa=0.2$} & \multicolumn{1}{| c|}{$\kappa=0.3$} & \multicolumn{1}{| c|}{$\kappa=0.4$} & \multicolumn{1}{| c|}{$\kappa=0.5$}\\
\cline{1-6}
\hline
3  & 1.1870& 1.4085 &  1.5756 & 1.6989 & 1.7470  \\
\cline{1-6}
\hline
4  & 1.7042 & 1.7933 & 1.8129 & 1.7709 & 1.7035 \\
\cline{1-6}
\hline
5  &  1.8423 & 1.8951 & 1.8884 & 1.7925 &  1.6070 \\
\cline{1-6}
\hline
6  & 1.9167 &  1.9433 & 1.9199 & 1.7856 & 1.4985 \\
\cline{1-6}
\hline
7  & 1.9587 & 1.9677 & 1.9371 &  1.7695 & 1.4104 \\
\cline{1-6}
\hline
8  & 1.9793 & 1.9790 &  1.9474  &  1.7514 & 1.3523 \\
\cline{1-6}
\hline
9  &  1.9865 & 1.9827 & 1.9537 & 1.7346 &  1.3182\\
\cline{1-6}
\hline
10 & 1.9868 & 1.9818 & 1.9568 &  1.7204  &  1.2989 \\
\cline{1-6}
\hline

\end{tabular}\label{rate4}
%}
\end{table}

\vspace{10cm}
\begin{table}[ht] %\tabcolsep0.04in
\centering
\def\arraystretch{1.20}
\caption{$H^1$ convergent rates for different gradient parameters $\kappa$ for consecutive mesh levels $j$.}
\tabcolsep=18pt%%
\vspace{-0.5\baselineskip}
%\resizebox{0.9\textwidth}{!}{
\begin{tabular}[c]{|c|c|c|c|c|c|}
\hline
$j  $& \multicolumn{1}{|c |}{$\kappa=0.1$}& \multicolumn{1}{|c |}{$\kappa=0.2$} & \multicolumn{1}{| c|}{$\kappa=0.3$} & \multicolumn{1}{| c|}{$\kappa=0.4$} & \multicolumn{1}{| c|}{$\kappa=0.5$}\\

\cline{1-6}
\hline
3  & 0.7343 & 0.5992 & 0.7016 &  0.8333 & 0.8501 \\
\cline{1-6}
\hline
4  & 0.9190 &  0.9052 & 0.8728 & 0.8742 & 0.8533 \\
\cline{1-6}
\hline
5  &   0.9362 & 0.9577 & 0.9375 & 0.9031 &  0.8319  \\
\cline{1-6}
\hline
6  & 0.9628 &  0.9759 & 0.9594 & 0.9124 &  0.7986 \\
\cline{1-6}
\hline
7  & 0.9811 &  0.9857 & 0.9693 &  0.9123 & 0.7629\\
\cline{1-6}
\hline
8  & 0.9904 & 0.9906 &  0.9747  &  0.9083 &   0.7305\\
\cline{1-6}
\hline
9  & 0.9938 & 0.9924 & 0.9779 & 0.9027 & 0.7041 \\
\cline{1-6}
\hline
10 & 0.9943 & 0.9924 & 0.9797 &  0.8964  & 0.6842\\
\cline{1-6}
\hline

\end{tabular}\label{rate3}
%}
\end{table}

\begin{figure}[H]
\centering
\subfigure[]{\includegraphics[width=0.49\textwidth]{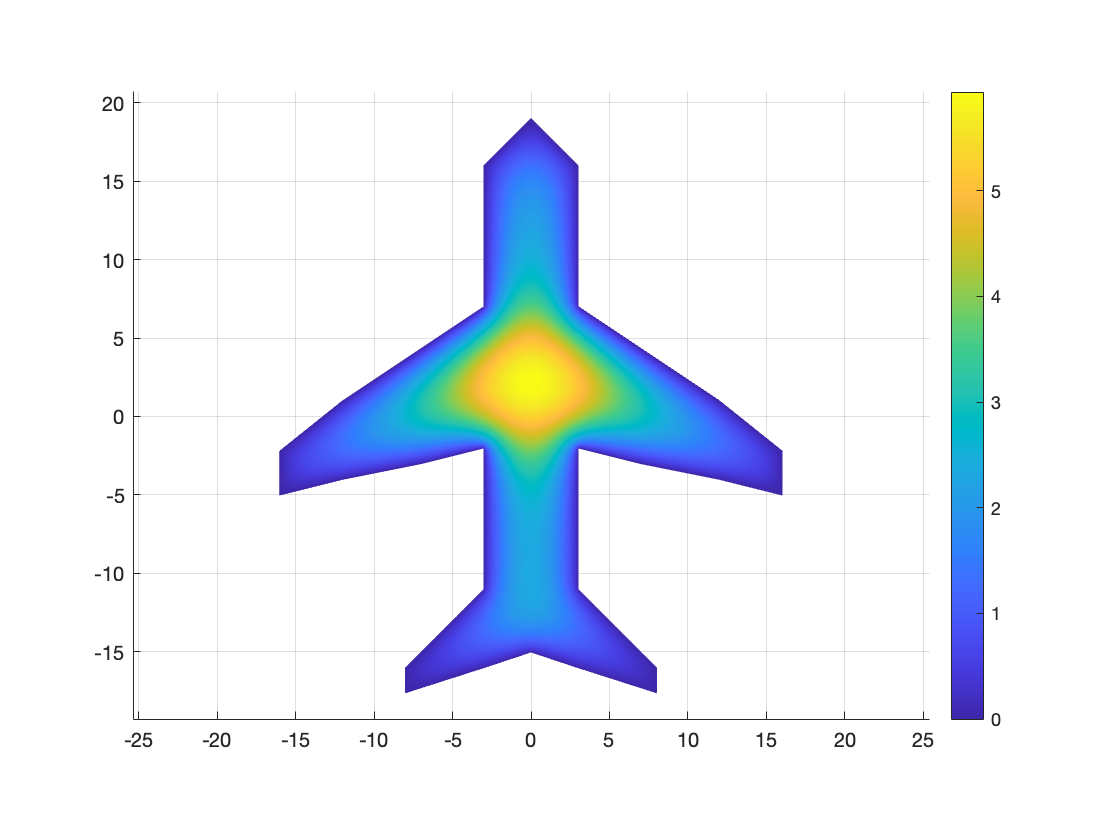}}
%\hspace{-0.5cm}
\subfigure[]{\includegraphics[width=0.49\textwidth]{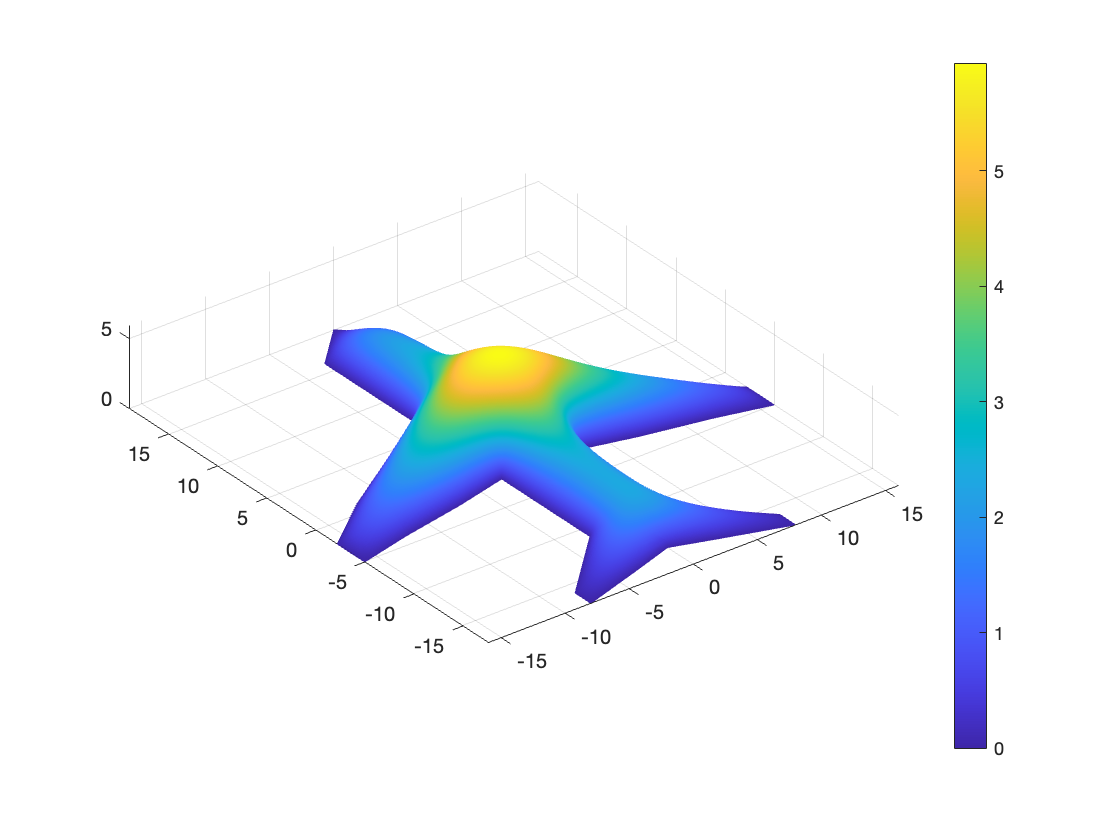}}

\caption{Numerical solution after 8 graded mesh refinements with $\kappa=0.1$.}\label{comp3}
\end{figure}\label{plan3d}

\section{Conclusion}

This work lays the groundwork for future research in solving more complex partial differential equations. It can also be used as a standard for evaluating the effectiveness of other numerical methods. We anticipate that it may be feasible to expand this method to solve the 3D Poisson equation, especially when dealing with singular solutions. This is currently the focus of our ongoing research. In summary, the proposed method offers a promising approach for efficiently and accurately solving elliptic partial differential equations, even when corner singularities are present.

\section*{Acknowledgments}
This research was financially supported by the Wayne State University.

\section*{Authors' Contributions}

All authors have read and approved the final version of the manuscript. The authors contributed equally to this work.

%The authors contributed to this work in the following ways:
%\begin{itemize}
%	\item {Author A:} 
%	\item {Author B:}
%\end{itemize}

\section*{Conflicts of Interest}
The authors declare that there are no conflicts of interest regarding the publication of this paper.

\end{document}